\documentclass[a4paper]{amsart}
\usepackage{amsmath}
\usepackage{amsfonts}

\newtheorem{theorem}{Theorem}
\theoremstyle{plain}
\newtheorem{acknowledgement}{Acknowledgement}

\newtheorem{corollary}{Corollary}

\newtheorem{example}{Example}

\newtheorem{lemma}{Lemma}

\input{tcilatex}

\begin{document}
\title[Inverse Sturm-Liouville problem ]{Inverse nodal problems for
Sturm-Liouville equation with nonlocal boundary conditions}
\author{A. Sinan Ozkan}
\curraddr{Department of Mathematics, Faculty of Science, Sivas Cumhuriyet
University 58140 Sivas, TURKEY}
\email{sozkan@cumhuriyet.edu.tr}
\author{\.{I}brahim Adalar}
\curraddr{Zara Veysel Dursun Colleges of Applied Sciences, Sivas Cumhuriyet
University Zara/Sivas, TURKEY}
\email{iadalar@cumhuriyet.edu.tr}
\subjclass[2010]{ [2010]34A55, 34B10, 34B24}
\keywords{Inverse nodal problem; Sturm-Liouville operator; Nonlocal boundary
condition}

\begin{abstract}
In this paper, a Sturm--Liouville problem with some nonlocal boundary
conditions of the Bitsadze-Samarskii type is studied. We show that the
coefficients of the problem can be uniquely determined by a dense set of
nodal points. Moreover, we give an algorithm for the reconstruction of the
potential function and some other coefficients in the boundary conditions.
\end{abstract}

\thanks{The author thanks to the reviewers for constructive comments and
recommendations which help to improve the readability and quality of the
paper.}
\maketitle

\section{\textbf{Introduction}}

The inverse nodal problem for a Sturm-Liouville operator consists in
reconstructing the operator from zeros of its eigenfunctions, namely nodal
points. This problem was studied firstly by McLaughlin in 1988 \cite{mc1}.
She showed that the potential of a Sturm-Liouville problem with Dirichlet
boundary conditions can be determined by a given dense subset of nodal
points. Immediately after, Hald and McLaughlin give some numerical schemes
for the reconstruction of the potential \cite{H}. In 1997, X.F. Yang gave a
solution algorithm of an inverse nodal problem for the Sturm-Liouville
operator with separated boundary conditions \cite{yang}. Inverse nodal
problems for Sturm-Liouville operators with the classical boundary
conditions have been studied in\ the papers (\cite{yangx}-\cite{guo2} ).

Nonlocal boundary conditions appear when we cannot measure data directly at
the boundary. This kind conditions arise in various some applied problems of
biology, biotechnology, physics and etc. As it is known there are two kinds
of nonlocal boundary conditions. One class of them is called integral type
conditions, and the other is the Bitsadze-Samarskii-type conditions.
Bitsadze and Samarskii are considered the originators of such conditions.
Nonlocal boundary conditions of the Bitsadze-Samarskii type were first
applied to elliptic equations by them \cite{bit}. Some important results on
the properties of eigenvalues and eigenfunctions of nonlocal boudary value
problems for Sturm-Liouville type operators have been published in various
publications (see, for example, \cite{stikonas1,stikonas2} and the
references therein).

Some inverse problems for a class of Sturm-Liouville operators with nonlocal
boundary conditions are investigated in \cite{niz,niz2}. In the literature,
there are only a few studies about inverse nodal problems with nonlocal
boundary conditions. Moreover all of them include integral type conditions.
Inverse nodal problems for this-type operators with different nonlocal
integral boundary are studied in (\cite{yang3}-\cite{yan6}). Especially,
C.F. Yang et all. solved inverse-nodal Sturm-Liouville problems with
nonlocal integral-type boundary conditions at only one or both end-points
(see \cite{yang4} and \cite{yan6}).

In the present paper, we consider Sturm-Liouville problems under some the
Bitsadze-Samarskii type nonlocal boundary conditions and obtain the
uniqueness of coefficients of the problem according to a set of nodal
points. Moreover, we give an algorithm for the reconstruction of these
coefficients.

Let us consider the following boundary value problem $L=L\left( q,h,H,\gamma
_{0},\gamma _{1},\xi _{0},\xi _{1}\right) :$%
\begin{equation}
\left. \ell y:=-y^{\prime \prime }+q(x)y=\lambda y,\text{ \ \ }x\in \Omega
=(0,1)\right.  \label{1}
\end{equation}%
\begin{equation}
\left. U(y):=y^{\prime }(0)+hy(0)-\gamma _{0}y(\xi _{0})=0,\right. \medskip
\label{2}
\end{equation}%
\begin{equation}
\left. V(y):=y^{\prime }(1)+Hy(1)-\gamma _{1}y(\xi _{1})=0,\right.  \label{3}
\end{equation}%
where $q(x)$ is a real valued continuously differentiable function; $h$, $H$ 
$\in 
\mathbb{R}
\cup \left\{ \infty \right\} $ and $\gamma _{i}\neq 0$ are real numbers for $%
i=0,1$; $\xi _{i}$ are rational numbers in $(0,1)$ for $i=0,1$ and $\lambda $
is the spectral parameter.

(2) and (3) are nonlocal conditions of a Bitsadze-Samarskii type. It is
clear that if $\xi _{0}=0$ and $\xi _{1}=1$, (\ref{2}) and (\ref{3}) are not
other than the classical separated boundary conditions. On the other hand,
while $\xi _{0}=1$ and $\xi _{1}=0$ (\ref{2}) and (\ref{3}) turn into
non-separated conditions. Inverse nodal problems for this type of boundary
conditions are studied by C.F. Yang \cite{yangg}. Therefore, we focus on the
case $\xi _{i}\in (0,1)$ in our investigation. In fact, since $\xi _{0}$ and 
$\xi _{1}$ are arbitrary rational numbers, the problem we are considering
involves a fairly large class of nonlocal boundary conditions.

The main goal of this paper is to solve inverse nodal problems for (\ref{1}%
)-(\ref{3}) in each of the following cases 
\begin{eqnarray*}
&&\left. \text{i) }h,\text{ }H\in 
\mathbb{R}
,\right. \medskip \\
&&\left. \text{ii) }h=\infty ,\text{ }H\in 
\mathbb{R}
,\right. \medskip \\
&&\left. \text{iii) }h\in 
\mathbb{R}
,\text{ }H=\infty .\right.
\end{eqnarray*}%
We note that if $h=\infty ,$ $H\in 
\mathbb{R}
,$ and $h\in 
\mathbb{R}
,$ $H=\infty $ the boundary conditions can be written as 
\begin{eqnarray*}
&&\left. y(0)=0,\right. \medskip \\
&&\left. y^{\prime }(1)+Hy(1)=\gamma _{1}y(\xi _{1})\right.
\end{eqnarray*}%
and 
\begin{eqnarray*}
&&\left. y^{\prime }(0)+hy(0)=\gamma _{0}y(\xi _{0}),\right. \medskip \\
&&\left. y(1)=0,\right.
\end{eqnarray*}%
respectively.

\section{\textbf{Spectral properties of the problem}}

Let $S(x,\lambda )$ and $C(x,\lambda )$ be the solutions of (\ref{1}) under
the initial conditions 
\begin{eqnarray*}
S(0,\lambda ) &=&0\text{, }S^{\prime }(0,\lambda )=1\medskip \\
C(0,\lambda ) &=&1\text{, }C^{\prime }(0,\lambda )=0\medskip
\end{eqnarray*}%
respectively. It can be calculated that $C(x,\lambda )$ and $S(x,\lambda )$
satisfy the following asymptotic relations for $\left\vert \lambda
\right\vert \rightarrow \infty $ (see \cite{yurko} and \cite{yang4})

\begin{equation}
C(x,\lambda )=\cos kx+\frac{\sin kx}{k}Q(x)+\frac{\cos kx}{k^{2}}%
q_{1}(x)+O\left( \dfrac{1}{k^{3}}\exp \left\vert \tau \right\vert x\right)
,\medskip  \label{4}
\end{equation}%
\begin{equation}
S(x,\lambda )=\frac{\sin kx}{k}-\frac{\cos kx}{k^{2}}Q(x)+O\left( \dfrac{1}{%
k^{3}}\exp \left\vert \tau \right\vert x\right) ,\medskip  \label{5}
\end{equation}%
where $\sqrt{\lambda }=k$, $\tau =\left\vert \func{Im}k\right\vert $, $Q(x)=%
\frac{1}{2}\int_{0}^{x}q(t)dt$ and $q_{1}(x)=\frac{q(x)-q(0)}{4}-\frac{1}{8}%
\left( \int_{0}^{x}q(t)dt\right) ^{2}.$

The characteristic function of problem (\ref{1})-(\ref{3})%
\begin{equation}
\Delta (\lambda )=\det \left( 
\begin{array}{cc}
U(C) & U(S)\medskip \\ 
V(C) & V(S)\medskip%
\end{array}%
\right)  \label{6}
\end{equation}%
and the zeros of the function $\Delta (\lambda )$ coincide with the
eigenvalues of the problem (\ref{1})-(\ref{3}). Clearly, $\Delta (\lambda )$
is entire function and so the problem has a discrete spectrum. \ 

Let $\left\{ \lambda _{n}\right\} _{n\geq 0}$ be the set of eigenvalues and $%
\varphi (x,\lambda _{n})$ be the eigenfunction corresponding to the
eigenvalue $\lambda _{n}.$ Some asymptotic formulas of $\lambda _{n}$ and $%
\varphi (x,\lambda _{n})$ are given in the following Lemmas.

\begin{lemma}
The numbers $\left\{ \lambda _{n}\right\} _{n\geq 0}$ are real for
sufficiently large $n$ and they satisfy the following asymptotic relation
for $n\rightarrow \infty $:%
\begin{equation*}
\sqrt{\lambda _{n}}=k_{n}=k_{n}^{0}+\frac{\kappa _{n}}{n\pi }+o(\frac{1}{n})
\end{equation*}%
$\bigskip $where $k_{n}^{0}=\left\{ 
\begin{array}{cc}
n\pi , & \text{if }h,\text{ }H\in 
\mathbb{R}
,\bigskip \\ 
\left( n+\frac{1}{2}\right) \pi , & 
\begin{array}{c}
\text{\ if }h=\infty ,\text{ }H\in 
\mathbb{R}
,\text{ } \\ 
\text{or }H=\infty ,\text{ }h\in 
\mathbb{R}
,%
\end{array}%
\end{array}%
\right. \bigskip $ and \newline
$\kappa _{n}=\left\{ 
\begin{array}{c}
\left. Q(1)+H-h-(-1)^{n}\left[ \gamma _{1}\cos \left( n\pi \xi _{1}\right)
-\gamma _{0}\cos \left( n\pi (1-\xi _{0})\right) \right] ,\text{ if }h,\text{
}H\in 
\mathbb{R}
,\right. \bigskip \\ 
\left. H+Q(1)-(-1)^{n}\gamma _{1}\sin \left( \left( n+\frac{1}{2}\right) \pi
\xi _{1}\right) ,\text{ \ \ \ \ \ \ \ \ \ \ \ \ \ \ \ \ \ \ \ \ if }h=\infty
,\text{ }H\in 
\mathbb{R}
,\right. \bigskip \\ 
\left. Q(1)-h+\gamma _{0}\cos \left( \left( n+\frac{1}{2}\right) \pi \xi
_{0}\right) ,\text{ \ \ \ \ \ \ \ \ \ \ \ \ \ \ \ \ \ \ \ \ \ \ if }h\in 
\mathbb{R}
,\text{ }H=\infty \right.%
\end{array}%
\right. $
\end{lemma}

\begin{proof}
\bigskip We give the proof for the case: $h,$ $H\in 
\mathbb{R}
$; the other cases are similar. From (\ref{6}), we have that 
\begin{eqnarray*}
\Delta (\lambda ) &=&hS^{\prime }(1,\lambda )-\gamma _{0}C(\xi _{0},\lambda
)S^{\prime }(1,\lambda )+HhS(1,\lambda )-\gamma _{0}HC(\xi _{0},\lambda
)S(1,\lambda )\bigskip \\
&&-h\gamma _{1}S(\xi _{1},\lambda )+\bigskip \gamma _{1}\gamma _{0}C(\xi
_{0},\lambda )S(\xi _{1},\lambda )-C^{\prime }(1,\lambda )+\gamma
_{0}C^{\prime }(1,\lambda )S(\xi _{0},\lambda ) \\
&&-HC(1,\lambda )+H\gamma _{0}S(\xi _{0},\lambda )C(1,\lambda )+\gamma
_{1}C(\xi _{1},\lambda )-\gamma _{1}\gamma _{0}S(\xi _{0},\lambda )C(\xi
_{1},\lambda ).\bigskip
\end{eqnarray*}%
Using (\ref{4}) and (\ref{5}), we obtain the following asymptotic formula
for $\Delta (\lambda )$ as $k\rightarrow \infty $:%
\begin{eqnarray*}
\Delta (\lambda ) &=&\left[ k\sin k-\cos kQ(1)+\frac{\sin k}{k}q_{1}(1)%
\right] +\gamma _{0}\left[ C^{\prime }(1,\lambda )S(\xi _{0},\lambda )-C(\xi
_{0},\lambda )S^{\prime }(1,\lambda )\right] \bigskip \\
&&+h\left[ \cos k+\frac{\sin k}{k}Q(1)\right] +Hh\left[ \frac{\sin k}{k}-%
\frac{\cos k}{k^{2}}Q(1)\right] \bigskip \\
&&-H\left[ \cos k+\frac{\sin k}{k}Q(1)+\frac{\cos k}{k^{2}}q_{1}(x)\right]
+H\gamma _{0}\left[ S(\xi _{0},\lambda )C(1,\lambda )-C(\xi _{0},\lambda
)S(1,\lambda )\right] \bigskip \\
&&+\gamma _{1}\gamma _{0}\left[ C(\xi _{0},\lambda )S(\xi _{1},\lambda
)\bigskip -S(\xi _{0},\lambda )C(\xi _{1},\lambda )\right] \bigskip \\
&&+\gamma _{1}\left[ \cos k\xi _{1}+\frac{\sin k\xi _{1}}{k}Q(1)+\frac{\cos
k\xi _{1}}{k^{2}}q_{1}(\xi _{1})\right] \bigskip \\
&&-\gamma _{1}\left[ h\frac{\sin k\xi _{1}}{k}-h\frac{\cos k\xi _{1}}{k^{2}}%
Q(\xi _{1})\right] +O\left( \dfrac{1}{k^{3}}\exp \left\vert \tau \right\vert
\right)
\end{eqnarray*}%
and so%
\begin{equation}
\Delta (\lambda )=k\sin k+w\cos k+\gamma _{1}\cos \left( k\xi _{1}\right)
-\gamma _{0}\cos k(1-\xi _{0})+O\left( \exp \left\vert \tau \right\vert
\right) ,\bigskip  \label{7}
\end{equation}%
where $w=h-H-Q(1).$ Let $G_{n}(\varepsilon )=\left\{ k:\left\vert k-n\pi
\right\vert <\varepsilon \right\} $ for $n=1,2,...$. It follows from (\ref{7}%
) that there exist some $M(\varepsilon )>0$ such that $\left\vert \Delta
(\lambda )\right\vert \geq M(\varepsilon )\left\vert k\right\vert \exp
\left\vert \tau \right\vert $ for sufficiently large $\left\vert
k\right\vert $ in $G_{n}(\varepsilon )$ \ Therefore $\lambda _{n}$ must be a
real number for sufficiently large $n$.

Moreover, if we apply Rouch\'{e} theorem to $h_{1}(\lambda )=k\sin k$ and $%
h_{2}(\lambda )=w\cos k+\gamma _{1}\cos \left( k\xi _{1}\right) -\gamma
_{0}\cos k(1-\xi _{0})+O\left( \exp \left\vert \tau \right\vert \right) $ on 
$\partial G_{n}(\varepsilon )$ for sufficiently small $\varepsilon ,$ we can
see that zeros of $\Delta (\lambda )$ satisfy%
\begin{equation*}
k_{n}=n\pi +\mu _{n},\text{ \ }\mu _{n}=o(1),\text{\ \ }n\rightarrow \infty .
\end{equation*}%
It follows from (7) that 
\begin{equation*}
\sin \left( n\pi +\mu _{n}\right) +O(\frac{1}{n})=0.
\end{equation*}%
Hence $\sin \left( \mu _{n}\right) =O(\frac{1}{n})$ and so $\mu _{n}=O(\frac{%
1}{n}).$ Thus%
\begin{equation}
k_{n}=n\pi +O(\frac{1}{n}),\text{\ \ }n\rightarrow \infty .  \label{8}
\end{equation}%
Using (7) and (\ref{8}) together, we get%
\begin{equation*}
\sin k_{n}+\frac{w}{n\pi }\cos k_{n}+\frac{\gamma _{1}}{n\pi }\cos \left(
k_{n}\xi _{1}\right) -\frac{\gamma _{0}}{n\pi }\cos \left( k_{n}(1-\xi
_{0})\right) +o(\frac{1}{n})=0.
\end{equation*}%
Therefore, we obtain 
\begin{equation}
\tan k_{n}=-\frac{w}{n\pi }-\frac{\gamma _{1}}{n\pi }\frac{\cos \left(
k_{n}\xi _{1}\right) }{\cos k_{n}}+\frac{\gamma _{0}}{n\pi }\frac{\cos
\left( k_{n}(1-\xi _{0})\right) }{\cos k_{n}}+o(\frac{1}{n}).  \label{9}
\end{equation}%
On the other hand, we have 
\begin{equation}
\frac{\cos \left( k_{n}\xi _{1}\right) }{n\pi \cos k_{n}}=(-1)^{n}\frac{\cos
\left( n\pi \xi _{1}\right) }{n\pi }+o(\frac{1}{n})  \label{10}
\end{equation}%
and%
\begin{equation}
\frac{\cos \left( k_{n}(1-\xi _{0})\right) }{n\pi \cos k_{n}}=(-1)^{n}\frac{%
\cos \left( n\pi (1-\xi _{0})\right) }{n\pi }+o(\frac{1}{n})  \label{11}
\end{equation}%
Using (\ref{10}) and (\ref{11}) in (\ref{9}), we get 
\begin{equation*}
\tan k_{n}=-\frac{w}{n\pi }-\gamma _{1}(-1)^{n}\frac{\cos \left( n\pi \xi
_{1}\right) }{n\pi }+\gamma _{0}(-1)^{n}\frac{\cos \left( n\pi (1-\xi
_{0})\right) }{n\pi }+o(\frac{1}{n}).
\end{equation*}%
Using Taylor's expansion of Arctangent, the proof can be concluded.
\end{proof}

It is clear that%
\begin{equation}
\varphi (x,\lambda _{n})=U(S(x,\lambda _{n}))C(x,\lambda _{n})-U(C(x,\lambda
_{n}))S(x,\lambda _{n})  \label{12}
\end{equation}%
From (\ref{12}) and Lemma 1, we can prove \ easily the following lemma:

\begin{lemma}
\bigskip The asymptotic formula%
\begin{equation}
\varphi (x,\lambda _{n})=\left\{ 
\begin{array}{c}
\left. \cos k_{n}x+\frac{\left( Q(x)-h\right) }{k_{n}}\sin k_{n}x+\frac{%
\gamma _{0}}{k_{n}}\sin k_{n}\left( x-\xi _{0}\right) +O\left( \dfrac{1}{%
k_{n}^{2}}\exp \left\vert \tau \right\vert x\right) ,\text{ \ for }h,\text{ }%
H\in 
\mathbb{R}
,\right. \bigskip \\ 
\left. \frac{\sin k_{n}x}{k_{n}}-\frac{\cos kx}{k_{n}^{2}}Q(x)+O\left( 
\dfrac{1}{k_{n}^{3}}\exp \left\vert \tau \right\vert x\right) ,\medskip 
\text{ \ for }h=\infty ,\text{ }H\in 
\mathbb{R}
,\right. \bigskip \\ 
\left. \frac{\sin k_{n}(1-x)}{k_{n}}+\frac{\cos k_{n}(1-x)}{k_{n}^{2}}%
Q(x)+O\left( \dfrac{1}{k_{n}^{3}}\exp \left\vert \tau \right\vert x\right) ,%
\text{ \ for }h\in 
\mathbb{R}
,\text{ }H=\infty \right.%
\end{array}%
\right.  \label{13}
\end{equation}%
is valid for sufficiently large $n.$
\end{lemma}

\section{\textbf{Inverse nodal problems: Uniqueness and reconstruction}}

We can see from Lemma 2 that $\varphi (x,\lambda _{n})$ has exactly $n-1$
nodal points in $\left( 0,1\right) .$ Let $X=\left\{ x_{n}^{j}:n=0,1,2,...%
\text{\ and }j=1,2,...,n-1\right\} $ be the set of nodal points. We assume
that $\int_{0}^{1}q(x)dx=0.$ Otherwise, the term $q(x)-\int_{0}^{1}q(x)dx$
is determined uniquely, instead of $q(x).$ \ 

\begin{lemma}
The elements of $X$ satisfy the following asymptotic formulas for
sufficiently large $n$,%
\begin{equation*}
x_{n}^{j}=\left\{ 
\begin{array}{c}
\dfrac{j+1/2}{n}+\dfrac{h-H+(-1)^{n}A_{n}}{n^{2}\pi ^{2}}\dfrac{\left(
j+1/2\right) }{n}+\dfrac{\left( Q(x_{n}^{j})-h\right) }{n^{2}\pi ^{2}}%
+\medskip \text{ \ \ \ \ \ \ \ \ \ \ \ \ \ \ \ \ \ \ \ \ \ \ \ \ \ \ \ \ \ \
\ \ \ \ \ \ \ \ \ \ \ \ } \\ 
\text{ \ \ \ \ \ \ \ \ \ \ \ \ \ \ \ \ \ \ \ \ \ \ \ \ \ \ \ \ \ \ \ \ }+%
\dfrac{\gamma _{0}}{n^{2}\pi ^{2}}\cos \left( n\pi \xi _{0}\right) +o\left( 
\frac{1}{n^{2}}\right) ,\bigskip \text{ \ \ \ \ \ \ if }h,H\in 
\mathbb{R}
\\ 
\dfrac{j}{n+\frac{1}{2}}-\dfrac{H-(-1)^{n}\gamma _{1}\sin \left( \left( n+%
\frac{1}{2}\right) \pi \xi _{1}\right) }{\left( n+\frac{1}{2}\right) ^{2}\pi
^{2}}\dfrac{j}{\left( n+\frac{1}{2}\right) }+\medskip \text{ \ \ \ \ \ \ \ \
\ \ \ \ \ \ \ \ \ \ \ \ \ \ \ \ \ \ \ \ \ \ \ \ \ \ \ \ \ \ \ \ \ \ } \\ 
\text{ \ \ \ \ \ \ \ \ \ \ \ \ \ \ \ \ \ \ \ \ \ \ \ \ \ \ \ \ \ \ \ \ \ \ }+%
\dfrac{Q(x_{n}^{j})}{\left( n+\frac{1}{2}\right) ^{2}\pi ^{2}}+o\left( \frac{%
1}{n^{2}}\right) ,\bigskip \text{ \ \ \ \ \ \ \ \ \ \ \ \ \ \ if }h=\infty
,H\in 
\mathbb{R}
\\ 
\dfrac{j+\frac{1}{2}}{n+\frac{1}{2}}+\left[ h-\gamma _{0}\cos \left( \left(
n+\frac{1}{2}\right) \pi \xi _{0}\right) \right] \dfrac{j+\frac{1}{2}}{%
\left( n+\frac{1}{2}\right) ^{3}\pi ^{2}}-\dfrac{h-Q(x_{n}^{j})}{\left( n+%
\frac{1}{2}\right) ^{2}\pi ^{2}}+\medskip \text{ \ \ \ \ \ \ \ \ \ \ \ \ }
\\ 
\text{ \ \ \ \ \ \ \ \ \ \ \ \ \ \ \ \ \ \ \ \ \ \ \ \ \ \ \ \ }+\gamma _{0}%
\dfrac{\cos \left( \left( n+\frac{1}{2}\right) \pi \xi _{0}\right) }{\left(
n+\frac{1}{2}\right) ^{2}\pi ^{2}}+o\left( \frac{1}{n^{2}}\right) ,\bigskip 
\text{ \ \ \ \ \ if }H=\infty ,h\in 
\mathbb{R}%
\end{array}%
\right.
\end{equation*}%
where $A_{n}=\left[ \gamma _{1}\cos \left( n\pi \xi _{1}\right) -\gamma
_{0}\cos \left( n\pi (1-\xi _{0})\right) \right] .$
\end{lemma}

\begin{proof}
\bigskip As before, we consider only the first case. One can obtain
similarly desired formulas for the other cases. Use the asymptotic formula (%
\ref{13}) to get%
\begin{equation*}
0=\varphi (x_{n}^{j},\lambda _{n})=\cos k_{n}x_{n}^{j}+\frac{\left(
Q(x_{n}^{j})-h\right) }{k_{n}}\sin k_{n}x_{n}^{j}+\frac{\gamma _{0}}{k_{n}}%
\sin k_{n}\left( x_{n}^{j}-\xi _{0}\right) +o\left( \dfrac{1}{k_{n}}\right)
\end{equation*}%
and so%
\begin{equation*}
\tan \left( k_{n}x_{n}^{j}-\frac{\pi }{2}\right) =\frac{\left(
Q(x_{n}^{j})-h\right) }{k_{n}}+\frac{\gamma _{0}}{k_{n}}\frac{\sin
k_{n}\left( x_{n}^{j}-\xi _{0}\right) }{\sin k_{n}x_{n}^{j}}+o\left( \dfrac{1%
}{k_{n}}\right) .
\end{equation*}%
This yields

\begin{equation}
x_{n}^{j}=\frac{\left( j+1/2\right) \pi }{k_{n}}+\frac{\left(
Q(x_{n}^{j})-h\right) }{k_{n}^{2}}+\frac{\gamma _{0}}{k_{n}^{2}}\frac{\sin
k_{n}\left( x_{n}^{j}-\xi _{0}\right) }{\sin k_{n}x_{n}^{j}}+o\left( \dfrac{1%
}{k_{n}^{2}}\right) .  \label{14}
\end{equation}%
Using $k_{n}x_{n}^{j}=\left( j+1/2\right) \pi +O(\frac{1}{n}),$ $%
n\rightarrow \infty $ we can show%
\begin{equation*}
\frac{\sin k_{n}\left( x_{n}^{j}-\xi _{0}\right) }{k_{n}^{2}\sin
k_{n}x_{n}^{j}}=\frac{\cos \left( n\pi \xi _{0}\right) }{n^{2}\pi ^{2}}%
+o\left( \frac{1}{n^{2}}\right) .
\end{equation*}%
On the other hand, we have 
\begin{eqnarray*}
\frac{1}{k_{n}} &=&\frac{1}{n\pi }\left( 1+\frac{w}{n^{2}\pi ^{2}}+\frac{%
(-1)^{n}}{n^{2}\pi ^{2}}A_{n}+o\left( \dfrac{1}{n^{3}}\right) \right) \\
\frac{1}{k_{n}^{2}} &=&\frac{1}{n^{2}\pi ^{2}}+o\left( \dfrac{1}{n^{3}}%
\right) .
\end{eqnarray*}%
using by Lemma 1. Therefore, it is concluded that,%
\begin{equation*}
x_{n}^{j}=\frac{j+1/2}{n}+\frac{h-H+(-1)^{n}A_{n}}{n^{2}\pi ^{2}}\frac{%
\left( j+1/2\right) }{n}+\frac{\left( Q(x_{n}^{j})-h\right) }{n^{2}\pi ^{2}}+%
\frac{\gamma _{0}}{n^{2}\pi ^{2}}\cos \left( n\pi \xi _{0}\right) +o\left( 
\frac{1}{n^{2}}\right) .
\end{equation*}
\end{proof}

According to Lemma 3 the existence of a dense subset $X_{0}$ of $X$ is
obvious.

\subsection{The Case $h,$\textbf{\ }$H\in 
\mathbb{R}
$}

Consider the problem $\widetilde{L}=L\left( \widetilde{q},\widetilde{h},%
\widetilde{H},\widetilde{\gamma }_{0},\widetilde{\gamma }_{1},\xi _{0},\xi
_{1}\right) $ under the same assumptions with $L.$ It is assumed in what
follows that if a certain symbol $s$ denotes an object related to the
problem $L$ then $\widetilde{s}$ denotes the coresponding object related to
the problem $\widetilde{L}$.

The following theorem is the first of our main results in this article.

\begin{theorem}[Uniqueness]
If $X_{0}=\widetilde{X}_{0}$ then $q(x)=\widetilde{q}(x)$ a.e. in $\left(
0,1\right) ,$ $h=\widetilde{h},$ $H=\widetilde{H},$ $\gamma _{0}=\widetilde{%
\gamma }_{0}$ and $\gamma _{1}=\widetilde{\gamma }_{1}.$ Thus, the potential 
$q(x),$ a.e. in $\left( 0,1\right) ,$ the coefficients $\gamma _{0},$ $%
\gamma _{1},$ $h$\ and $H$ are uniquely determined by $X_{0}$.
\end{theorem}

\begin{proof}
\textbf{Step 1.} Put $\xi _{0}=\dfrac{p_{0}}{r_{0}}$ and $\xi _{1}=\dfrac{%
p_{1}}{r_{1}}$, where $p_{i},r_{i}\in 
\mathbb{Z}
$ for $i=0,1.$ For each fixed $x\in \left[ 0,1\right] $, there exists a
sequence $\left( x_{n}^{j}\right) $ converges to $x.$ Clearly the
subsequence $\left( x_{m}^{j}\right) $ converges also to $x$ for $%
m=2r_{0}r_{1}n.$ On the other hand, $\underset{m\rightarrow \infty }{\lim }%
A_{m}=\gamma _{1}-\gamma _{0}.$ Therefore we can see from Lemma 3 the
following limit exists and given equality holds:%
\begin{equation}
\left. \underset{m\rightarrow \infty }{\lim }m^{2}\pi ^{2}\left( x_{m}^{j}-%
\frac{j}{m}\right) =f(x)=\left( h-H+\gamma _{1}-\gamma _{0}\right)
x+Q(x)-h+\gamma _{0},\right.  \label{15}
\end{equation}%
Direct calculations in (\ref{15}) yield 
\begin{eqnarray*}
&&\left. \gamma _{0}-h=f(0),\medskip \right. \\
&&\left. \gamma _{1}-H=f(1),\right. \medskip \\
&&\left. q(x)=2\left( f^{\prime }(x)+f(0)-f(1)\right) .\right.
\end{eqnarray*}

Since $X_{0}=\widetilde{X}_{0}$ then $f(x)=\widetilde{f}(x)$ and so $q(x)=%
\widetilde{q}(x)$, a.e. in $\left( 0,1\right) .$

\textbf{Step 2.} To show $\widetilde{h}=h\ $and $\gamma _{0}=\widetilde{%
\gamma }_{0}$ consider a sequence $\left\{ x_{n}^{j}\right\} \subset X_{0}$
converges to $\xi _{0}$ and write the equation (\ref{1}) for $\varphi
(x,\lambda _{n})$ and $\widetilde{\varphi }(x,\widetilde{\lambda }_{n})$:

\begin{eqnarray*}
-\widetilde{\varphi }^{\prime \prime }\left( x,\widetilde{\lambda }%
_{n}\right) +q(x)\widetilde{\varphi }\left( x,\widetilde{\lambda }%
_{n}\right) &=&\widetilde{\lambda }_{n}\widetilde{\varphi }\left( x,%
\widetilde{\lambda }_{n}\right) ,\bigskip \\
-\varphi ^{\prime \prime }\left( x,\lambda _{n}\right) +q(x)\varphi \left(
x,\lambda _{n}\right) &=&\lambda _{n}\varphi \left( x,\lambda _{n}\right) .
\end{eqnarray*}%
If we apply the procedure:

(i): multiplied by $\varphi \left( x,\lambda _{n}\right) $ and $\widetilde{%
\varphi }\left( x,\widetilde{\lambda }_{n}\right) ,$ respectively; (ii):
subtracted from each other and (iii): integrated over the interval $\left(
\xi _{0},x_{n}^{j}\right) $ the equality{}%
\begin{equation*}
\varphi ^{\prime }\left( \xi _{0},\lambda _{n}\right) \widetilde{\varphi }%
\left( \xi _{0},\widetilde{\lambda }_{n}\right) -\widetilde{\varphi }%
^{\prime }\left( \xi _{0},\widetilde{\lambda }_{n}\right) \varphi \left( \xi
_{0},\lambda _{n}\right) =\left( \widetilde{\lambda }_{n}-\lambda
_{n}\right) \int\limits_{\xi _{0}}^{x_{n}^{j}}\widetilde{\varphi }\left( x,%
\widetilde{\lambda }_{n}\right) \varphi \left( x,\lambda _{n}\right) dx
\end{equation*}%
is obtained. From Lemma 1 the following estimate holds for sufficiently
large $n$ 
\begin{equation}
\varphi ^{\prime }\left( \xi _{0},\lambda _{n}\right) \widetilde{\varphi }%
\left( \xi _{0},\widetilde{\lambda }_{n}\right) -\widetilde{\varphi }%
^{\prime }\left( \xi _{0},\widetilde{\lambda }_{n}\right) \varphi \left( \xi
_{0},\lambda _{n}\right) =o(1),\text{ }n\rightarrow \infty .  \label{16}
\end{equation}%
Using (\ref{16}) and Lemma 2 we get%
\begin{equation*}
\left[ \varphi ^{\prime }(\xi _{0},\lambda _{n})-\widetilde{\varphi }%
^{\prime }\left( \xi _{0},\widetilde{\lambda }_{n}\right) \right] \cos n\pi
\xi _{0}=o(1),\text{ }n\rightarrow \infty .
\end{equation*}%
The last equality yields%
\begin{equation*}
\left( \widetilde{h}-h\right) \cos ^{2}n\pi \xi _{0}+\left( \gamma _{0}-%
\widetilde{\gamma }_{0}\right) \cos n\pi \xi _{0}=o(1),\text{ }n\rightarrow
\infty .
\end{equation*}%
Therefore, we conclude that $\widetilde{h}=h\ $and $\gamma _{0}=\widetilde{%
\gamma }_{0}.$

\textbf{Step 3. }Finally let us prove $\gamma _{1}=\widetilde{\gamma }_{1}$
and $H=\widetilde{H}.$ Consider another sequence $\left\{ x_{n}^{j}\right\}
\subset X_{0}$ converges to $\xi _{1}.$ If we apply above procedure but take
the integral from $\xi _{1}$ to $x_{n}^{j}$, we get$\bigskip $ 
\begin{equation*}
\left. \widetilde{\varphi }^{\prime }\left( \xi _{1},\widetilde{\lambda }%
_{n}\right) \varphi \left( \xi _{1},\lambda _{n}\right) -\varphi ^{\prime
}\left( \xi _{1},\lambda _{n}\right) \widetilde{\varphi }\left( \xi _{1},%
\widetilde{\lambda }_{n}\right) =o(1),\text{ }n\rightarrow \infty \right.
\end{equation*}%
instead of (\ref{16}). From 3, we have%
\begin{equation*}
\left. \widetilde{\varphi }^{\prime }\left( \xi _{1},\widetilde{\lambda }%
_{n}\right) \left[ \frac{\varphi ^{\prime }\left( 1,\lambda _{n}\right)
+H\varphi \left( 1,\lambda _{n}\right) }{\gamma _{1}}\right] -\varphi
^{\prime }\left( \xi _{1},\lambda _{n}\right) \left[ \frac{\widetilde{%
\varphi }\left( 1,\widetilde{\lambda }_{n}\right) +\widetilde{H}\widetilde{%
\varphi }\left( 1,\widetilde{\lambda }_{n}\right) }{\widetilde{\gamma }_{1}}%
\right] =o(1),\text{ }n\rightarrow \infty .\right. \bigskip
\end{equation*}%
Using Lemma1 and Lemma 2, it can be calculated that $\bigskip $%
\begin{equation*}
\sin n\pi \xi _{1}\left[ \left( \frac{H-h}{\gamma _{1}}-\frac{\widetilde{H}-h%
}{\widetilde{\gamma }_{1}}\right) \frac{(-1)^{n}}{n\pi }+\frac{\gamma _{0}}{%
n\pi }\left( \frac{1}{\gamma _{1}}-\frac{1}{\widetilde{\gamma }_{1}}\right)
\cos n\pi \left( 1-\xi _{0}\right) \right] =o(\frac{1}{n}),\text{ }%
n\rightarrow \infty .
\end{equation*}%
This yields%
\begin{equation*}
\sin n\pi \xi _{1}\left[ (-1)^{n}\left( \frac{H-h}{\gamma _{1}}-\frac{%
\widetilde{H}-h}{\widetilde{\gamma }_{1}}\right) +\gamma _{0}\left( \frac{1}{%
\gamma _{1}}-\frac{1}{\widetilde{\gamma }_{1}}\right) \cos n\pi \left( 1-\xi
_{0}\right) \right] =o(1).
\end{equation*}%
Hence $\gamma _{1}=\widetilde{\gamma }_{1}$ and $H=\widetilde{H}.$ This
completes the proof.
\end{proof}

\begin{corollary}[Reconstruction algorithm]
Let $X_{0}$, $\xi _{0}=\dfrac{p_{0}}{r_{0}}$ and $\xi _{1}=\dfrac{p_{1}}{%
r_{1}}$ be given. Then $q(x),$ $\gamma _{0}-h$ and $\gamma _{1}-H$ can be
reconstructed by the following algorithm:\newline
i) Denote $m=2r_{0}r_{1}n;$\newline
ii) Find $f(x)$ by (\ref{15});\newline
iii) Find $q(x),$ $\gamma _{0}-h$ and $\gamma _{1}-H$ by the formulas 
\begin{eqnarray*}
&&\left. q(x)=2\left( f^{\prime }(x)+f(0)-f(1)\right) \right. \medskip \\
&&\left. \gamma _{0}-h=f(0),\medskip \right. \\
&&\left. \gamma _{1}-H=f(1).\right. \medskip
\end{eqnarray*}%
Note that if one of the pairs $\left( h,H\right) $ and $\left( \gamma
_{0},\gamma _{1}\right) $ is given, we can find the other pair.
\end{corollary}

\begin{example}
Consider the nonlocal BVP%
\begin{equation*}
L:\left\{ 
\begin{array}{c}
\left. \ell y:=-y^{\prime \prime }+q(x)y=\lambda y,\text{ \ \ }x\in
(0,1)\right. \medskip \\ 
\left. U(y):=y^{\prime }(0)+hy(0)=\gamma _{0}y(\frac{2}{5}),\right. \medskip
\\ 
\left. V(y):=y^{\prime }(1)+Hy(1)=\gamma _{1}y(\frac{6}{7}),\right. \medskip%
\end{array}%
\right.
\end{equation*}%
where $q(x)\in C^{1}\left[ 0,1\right] ,$ $\gamma _{0},$ $\gamma _{1},$ $h,$
and $H$ $\in 
\mathbb{R}
$ are unknown coefficients. Let $\ X_{0}=\left\{ x_{n}^{j}\right\} $ be the
given subset of nodal points which satisfy the following asymptotics%
\begin{eqnarray*}
x_{n}^{j} &=&\frac{\left( j+1/2\right) }{n}+\frac{-1+(-1)^{n}\left[ 6\cos
\left( \frac{6n\pi }{7}\right) -3\cos \left( \frac{3n\pi }{5})\right) \right]
}{n^{2}\pi ^{2}}\frac{\left( j+1/2\right) }{n}+\frac{\left( \sin \frac{%
\left( j+1/2\right) \pi }{n}-2\pi \right) }{2n^{2}\pi ^{3}}\medskip \\
&&+\frac{3}{n^{2}\pi ^{2}}\cos \left( 2n\frac{\pi }{5}\right) +o\left( \frac{%
1}{n^{2}}\right)
\end{eqnarray*}%
\newline
Let $m:=70n.$ One can calculate that, 
\begin{equation*}
\left. \underset{m\rightarrow \infty }{\lim }m^{2}\pi ^{2}\left( x_{m}^{j}-%
\frac{j}{m}\right) =f(x)=2x+\left( \frac{\sin \pi x-2\pi }{2\pi }\right)
+3\right.
\end{equation*}%
According to Theorem 1, we find 
\begin{equation*}
\left. q(x)=2\left( f^{\prime }(x)+f(0)-f(1)\right) =\cos \pi x.\right.
\end{equation*}%
and%
\begin{eqnarray*}
&&\left. \gamma _{0}-h=f(0)=2,\medskip \right. \\
&&\left. \gamma _{1}-H=f(1)=4\right. \medskip
\end{eqnarray*}%
If the pair $\left( h,H\right) $ is given as, for example, $h=1$ and\ $H=2$
then we find $\gamma _{0}=3$ and $\gamma _{1}=6$.
\end{example}

\subsection{The Case $h=\infty ,$ $H\in 
\mathbb{R}
$}

In this subsection, we consider the equation (\ref{1}) with one Dirichlet
boundary condition%
\begin{equation}
\left. U(y):=y(0)=0\right. \medskip  \label{17}
\end{equation}%
and with the nonlocal boundary condition (\ref{3}).

Let $X_{0}$ be a dense nodal points-set. For each fixed $x$ in $\left(
0,1\right) ,$ it can be choosen a sequence $\left( x_{n}^{j}\right) \subset
X_{0}$ which converges to $x.$ Therefore we can show from Lemma 3 that the
following limit exists and finite for $m=2r_{1}n$:%
\begin{eqnarray*}
&&\left. \underset{m\rightarrow \infty }{\lim }\left( m+\frac{1}{2}\right)
^{2}\pi ^{2}\left( x_{m}^{j}-\frac{j}{m+\frac{1}{2}}\right) =g(x)\right.
\medskip \\
&&\text{ \ \ \ \ \ \ \ \ \ \ \ \ \ \ \ \ \ \ \ \ \ \ \ \ \ \ \ \ \ \ \ \ \ }%
\left. =(\gamma _{1}\sin \frac{\pi }{2}\xi _{1}-H)x+\dfrac{1}{2}%
\int_{0}^{x}q(t)dt.\right.
\end{eqnarray*}%
Thus, we can prove the following theorem using methods similar to one in the
proof of Theorem 1.

\begin{theorem}
If $X_{0}=\widetilde{X}_{0}$ then $q(x)=\widetilde{q}(x)$ a.e. in $\left(
0,1\right) $, $H=\widetilde{H},$ and $\gamma _{1}=\widetilde{\gamma }_{1}.$
Moreover if $X_{0}$ and $\xi _{1}=\dfrac{p_{1}}{r_{1}}$ is given, $q(x)$ and 
$\gamma _{1}\sin \frac{\pi }{2}\xi _{1}-H$ can be reconstructed by the
following formulas:%
\begin{eqnarray*}
&&\left. q(x)=2\left( g^{\prime }(x)-g(1)\right) ,\right. \medskip \\
&&\left. \gamma _{1}\sin \frac{\pi }{2}\xi _{1}-H=g(1).\right. \medskip
\end{eqnarray*}
\end{theorem}

\begin{example}
Consider the nonlocal BVP%
\begin{equation*}
L:\left\{ 
\begin{array}{c}
\left. -y^{\prime \prime }+q(x)y=\lambda y,\text{ \ \ }x\in (0,1)\right.
\medskip \\ 
\left. y(0)=0,\right. \medskip \\ 
\left. y^{\prime }(1)+2y(1)-\gamma _{1}y(\frac{2}{5})=0,\right.%
\end{array}%
\right.
\end{equation*}%
where $q(x)\in C^{1}\left[ 0,1\right] $ and $\gamma _{1}$ are unknown\ real
coefficients. Let $\ X_{0}=\left\{ x_{n}^{j}\right\} $ be the given subset
of nodal points which satisfy the following asymptotics 
\begin{eqnarray*}
x_{n}^{j} &=&\dfrac{j}{n+\frac{1}{2}}-\dfrac{2-(-1)^{n}3\sin \left( \frac{%
2n+1}{5}\pi \right) }{\left( n+\frac{1}{2}\right) ^{2}\pi ^{2}}\dfrac{j}{%
\left( n+\frac{1}{2}\right) }\medskip + \\
&&+\dfrac{1}{2\left( n+\frac{1}{2}\right) ^{2}\pi ^{3}}\left( \sin \frac{%
j\pi }{n+\frac{1}{2}}+\frac{j\pi }{n+\frac{1}{2}}\left( \frac{j}{2n+1}-\frac{%
1}{2}\right) \right) +o\left( \frac{1}{n^{2}}\right) \bigskip \text{.}
\end{eqnarray*}%
\newline
To find $q(x)$ and $\gamma _{1}$ we take $m=10n$ and calculate the following
lmit%
\begin{equation*}
\left. \underset{m\rightarrow \infty }{\lim }\left( m+\frac{1}{2}\right)
^{2}\pi ^{2}\left( x_{m}^{j}-\frac{j}{\left( m+\frac{1}{2}\right) }\right)
=g(x)=\left( 3\sin \frac{\pi }{5}-2\right) x+\frac{\sin \pi x}{2\pi }+\frac{x%
}{2}\left( \frac{x}{2}-\frac{1}{2}\right) \right.
\end{equation*}%
Thus, we find 
\begin{eqnarray*}
&&\left. q(x)=2\left( g^{\prime }(x)-g(1)\right) =\cos \pi x+x-\frac{1}{2}%
\right. \medskip \\
&&\left. \gamma _{1}=\frac{g(1)+2}{\sin \frac{\pi }{5}}=3\right. \medskip
\end{eqnarray*}
\end{example}

\subsection{The Case $H=\infty ,$ $h\in 
\mathbb{R}
$}

In this subsection, we consider the equation (\ref{1}) with nonlocal
boundary condition (\ref{2}) and one Dirichlet boundary condition%
\begin{equation}
\left. V(y):=y(1)=0,\right. \medskip  \label{18}
\end{equation}

Let $m:=2r_{0}n.$ Here, $r_{0}$ denotes the denominators of $\xi _{0}.$

Let $X_{0}$ be a dense nodal points-set. For each fixed $x$ in $\left(
0,1\right) ,$ it can be choosen a sequence $\left( x_{n}^{j}\right) \subset
X_{0}$ which converges to $x.$ Therefore we can show from Lemma 3 that%
\begin{eqnarray*}
&&\left. \underset{m\rightarrow \infty }{\lim }\left( m+\frac{1}{2}\right)
^{2}\pi ^{2}\left( x_{m}^{j}-\frac{j+\frac{1}{2}}{m+\frac{1}{2}}\right)
=\psi (x)\right. \medskip \\
&&\left. \text{ \ \ \ \ \ \ \ \ \ \ \ \ \ \ \ \ \ \ \ \ \ \ }=(h-\gamma
_{0}\cos \frac{\pi }{2}\xi _{0})x-h+\gamma _{0}\cos \frac{\pi }{2}\xi _{0}+%
\dfrac{1}{2}\int_{0}^{x}q(t)dt.\right.
\end{eqnarray*}

\bigskip Thus,we can give the following theorem.

\begin{theorem}
If $X_{0}=\widetilde{X}_{0}$ then $q(x)=\widetilde{q}(x)$ a.e. in $\left(
0,1\right) $, $h=\widetilde{h}$ and $\gamma _{0}=\widetilde{\gamma }_{0}.$
Moreover if $X_{0}$ and $\xi _{0}=\dfrac{p_{0}}{r_{0}}$ are given, $q(x)$
and $\widetilde{\gamma }_{0}\cos \frac{\pi }{2}\xi _{0}-h$ can be
reconstructed by the following formulae:%
\begin{equation*}
\left. 
\begin{array}{c}
q(x)=2\left( \psi ^{\prime }(x)+\psi (0)\right) ,\medskip \\ 
\gamma _{0}\cos \frac{\pi }{2}\xi _{0}-h=\psi (0)%
\end{array}%
\right.
\end{equation*}
\end{theorem}

\begin{example}
Consider the nonlocal BVP%
\begin{equation*}
L:\left\{ 
\begin{array}{c}
\left. \ell y:=-y^{\prime \prime }+q(x)y=\lambda y,\text{ \ \ }x\in
(0,1)\right. \medskip \\ 
\left. U(y):=y^{\prime }(0)+y(0)=\gamma _{0}y(\frac{2}{3}),\right. \medskip
\\ 
\left. V(y):=y(1)=0,\right. \medskip%
\end{array}%
\right.
\end{equation*}%
where $q(x)\in C^{1}\left[ 0,1\right] $ and $\gamma _{0}$ are unknown
coefficients. Let $\ X_{0}=\left\{ x_{n}^{j}\right\} $ be the given subset
of nodal points which satisfy the following asymptotics%
\begin{eqnarray*}
x_{n}^{j} &=&\frac{j+\frac{1}{2}}{n+\frac{1}{2}}+\left[ 1-2\cos \left( \frac{%
2n+1}{3}\pi \right) \right] \frac{j+\frac{1}{2}}{\left( n+\frac{1}{2}\right)
^{3}\pi ^{2}}-\frac{1}{\left( n+\frac{1}{2}\right) ^{2}\pi ^{2}}\medskip \\
&&+\frac{2\cos \left( \frac{2n+1}{3}\pi \right) }{\left( n+\frac{1}{2}%
\right) ^{2}\pi ^{2}}-\frac{\cos \frac{\left( j+\frac{1}{2}\right) \pi }{n+%
\frac{1}{2}}-1}{2\left( n+\frac{1}{2}\right) ^{2}\pi ^{3}}-\frac{j+\frac{1}{2%
}}{\left( n+\frac{1}{2}\right) ^{3}\pi ^{3}}+o\left( \frac{1}{n^{2}}\right)
.\ \medskip
\end{eqnarray*}%
\newline
Let $m:=6n.$ One can calculate that, 
\begin{equation*}
\left. \underset{m\rightarrow \infty }{\lim }\left( m+\frac{1}{2}\right)
^{2}\pi ^{2}\left( x_{m}^{j}-\frac{j}{m}\right) =\psi (x)=\left( 1-2\cos 
\frac{\pi }{3}\right) x-1+2\cos \frac{\pi }{3}-\frac{\cos \pi x-1}{2\pi }-%
\frac{x}{\pi }.\right.
\end{equation*}%
According to Theorem 3, we find 
\begin{eqnarray*}
&&\left. q(x)=2\left( \psi ^{\prime }(x)+\psi (0)\right) =\sin \pi x-\frac{2%
}{\pi },\right. \\
&&\left. \gamma _{0}=2\psi (0)+2=2.\medskip \right. \medskip
\end{eqnarray*}
\end{example}

\begin{acknowledgement}
The authors would like to thank the referees for their valuable comments
which helped to improve the manuscript.
\end{acknowledgement}

\end{document}